\documentclass[11pt]{amsart}

\usepackage{amsmath}
\usepackage{fullpage}
\usepackage{xspace}
\usepackage[psamsfonts]{amssymb}
\usepackage[latin1]{inputenc}
\usepackage{graphicx,color}
\usepackage[curve]{xypic} 
\usepackage{hyperref}
\usepackage{graphicx}

\usepackage{amsmath}%
\usepackage{amsthm}%
\usepackage{amscd}
\usepackage{amsfonts}%
\usepackage{amssymb}%
\usepackage{graphicx}

\usepackage{mathrsfs}

\usepackage{tikz}
\usetikzlibrary{matrix,arrows}

\usepackage{tikz-cd}

\newtheorem{theorem}{Theorem}[section]

\newtheorem{conjecture}[theorem]{Conjecture}
\newtheorem{corollary}[theorem]{Corollary}

\newtheorem{lemma}[theorem]{Lemma}

\theoremstyle{remark}

\numberwithin{equation}{section}

\newcommand{\gfrak}{\mathfrak{g}}

\newcommand{\Rcal}{\mathscr{R}}

\newcommand{\Pro}{\mathbb{P}}

\newcommand{\Z}{\mathbb{Z}}
\newcommand{\C}{\mathbb{C}}

\newcommand{\Q}{\mathbb{Q}}

\newcommand{\N}{\mathbb{N}}
\newcommand{\A}{\mathbb{A}}

\newcommand{\ord}{\mathrm{ord}}

\newcommand{\rk}{\mathrm{rank}\,}

\newcommand{\alg}{\mathrm{alg}}
\newcommand{\Rat}{\mathrm{Rat}}

  \DeclareFontFamily{U}{wncy}{}
    \DeclareFontShape{U}{wncy}{m}{n}{<->wncyr10}{}
    \DeclareSymbolFont{mcy}{U}{wncy}{m}{n}
    \DeclareMathSymbol{\Sha}{\mathord}{mcy}{"58}

\begin{document}
\title[A diophantine definition of the constants]{A Diophantine definition of the constants in $\mathbb{Q}(z)$}

\author{Natalia Garcia-Fritz}
\address{ Departamento de Matem\'aticas,
Pontificia Universidad Cat\'olica de Chile.
Facultad de Matem\'aticas,
4860 Av.\ Vicu\~na Mackenna,
Macul, RM, Chile}
\email[N. Garcia-Fritz]{natalia.garcia@mat.uc.cl}%

\author{Hector Pasten}
\address{ Departamento de Matem\'aticas,
Pontificia Universidad Cat\'olica de Chile.
Facultad de Matem\'aticas,
4860 Av.\ Vicu\~na Mackenna,
Macul, RM, Chile}
\email[H. Pasten]{hector.pasten@mat.uc.cl}%

\thanks{N.G.-F. was supported by ANID Fondecyt Regular grant 1211004 from Chile. H.P was supported by ANID (ex CONICYT) FONDECYT Regular grant 1190442 from Chile. Both authors were supported by the NSF Grant No. DMS-1928930 while at the MSRI, Berkeley, in Summer 2022.}
\date{\today}
\subjclass[2010]{Primary: 11U09; Secondary: 14J27, 11G05} %
\keywords{Diophantine sets, rational functions, elliptic surfaces}%

\begin{abstract} It is an old problem in the area of Diophantine definability to determine whether $\mathbb{Q}$ is Diophantine in $\mathbb{Q}(z)$. We provide a positive answer conditional on two standard conjectures on elliptic surfaces.
\end{abstract}

\maketitle



\section{Introduction}


Let $A$ be a ring. A set $S\subseteq A^m$ is \emph{Diophantine} over $A$ if we have  $F_1,...,F_r\in A[x_1,...,x_m,y_1,...,y_n]$ such that for each ${\bf a}=(a_1,...,a_m)\in A^m$ the following holds: ${\bf a}\in S$ if and only if the equations $F_j({\bf a},{\bf y})=0$ for $1\le j\le r$ have a common solution in $A^n$. Equivalently, $S$ is Diophantine if it is positive existentially definable over the language of rings with parameters from $A$.

Let $z$ be transcendental over $\Q$. R.\ Robinson \cite{RR} proved in the sixties that $\Q$ is first order definable in the field $\Q(z)$, and since then the question of whether $\Q$ is Diophantine in  $\Q(z)$ remains open. It is known \cite{KoeLarge, FehmGeyer} that if $L$ is a large field or if $L^\times / (L^\times)^n$ is finite for some $n\ge 2$ coprime to ${\rm char}(L)$, then $L$ is Diophantine in $L(z)$ (see also Section 6.5 in \cite{DaansThesis}), but unfortunately this gives no indication on what to expect for $\Q$ in $\Q(z)$. In this note we prove a positive answer conditional on some standard conjectures on elliptic surfaces. For us, elliptic surfaces have a distinguished section by definition; see Section \ref{SecEll}.

Before stating our main result let us formulate the relevant conjectures. The first one concerns the rank of fibres in an elliptic surface and it originates in the work of Helfgott \cite{Helfgott}; see Section \ref{SecPositive}. 
\begin{conjecture}[Positivity of the rank]\label{Conj1} Let $\pi\colon X\to \Pro^1$ be an elliptic surface defined over $\Q$ and fix an affine chart $\A^1\subseteq \Pro^1$. Suppose that the elliptic surface has a fibre of multiplicative type over some point of $\A^1$. Let us define
$$
N(x)=\#\{n\in \Z : 1\le n\le x\mbox{ and the fibre }X_n\mbox{ is an elliptic curve with }\rk X_n(\Q)>0\}.
$$
Then there is $c>0$ such that for all sufficiently large $x$ we have $N(x)\ge c\cdot x$.
\end{conjecture}

The second one is a conjecture of Ulmer \cite{Ulmer} that gives a bound on the degree of multisections of elliptic surfaces, provided that the Kodaira dimension is $1$. See Section \ref{SecUlmer} for details.
\begin{conjecture}[Boundedness of the degree of multisections]\label{Conj2} Let $\pi\colon X\to \Pro^1$ be an elliptic surface over $\C$ with $X$ a smooth projective surface of Kodaira dimension $1$, and let $F$ be a fibre of $\pi$. There is a number $M$ such that for every (possibly singular) rational curve $C\subseteq X$ not contained in a fibre, we have that $(C.F)\le M$. 
\end{conjecture}
Our main result is
\begin{theorem}\label{ThmMain} If Conjectures \ref{Conj1} and \ref{Conj2} are true, then $\Q$ is Diophantine in $\Q(z)$. Furthermore, in this case every listable subset $T\subseteq \Q$ is Diophantine over $\Q(z)$.
\end{theorem}
Here we recall that $T\subseteq \Q$ is \emph{listable} if there is a Turing machine that produces all the elements of $T$ and nothing else (possibly with repetitions, in some order).

Conjecture \ref{Conj1} is essential in our approach. In Section \ref{SecPositive} it is observed that it follows from certain conjecture strongly supported by the work of Helfgott \cite{Helfgott} and a small part of the Birch and Swinnerton-Dyer conjecture. 

On the other hand, what we really need from Conjecture \ref{Conj2} is the following consequence (see Theorem \ref{ThmV2}), which perhaps can be proved unconditionally by other means:

\begin{conjecture}\label{ConjAdHoc} There is an elliptic surface $\pi\colon X\to \Pro^1$ defined over $\Q$ satisfying the following two conditions:
\begin{itemize}
\item[(i)] $\pi$ has a fibre of multiplicative reduction defined over $\C$;
\item[(ii)] $X$ contains at most finitely many (possibly singular) rational curves defined over $\C$.
\end{itemize}
\end{conjecture}

Here we stress the fact that just \emph{one} example of such a surface suffices for our approach. That Conjecture \ref{ConjAdHoc} follows from Conjecture \ref{Conj2} will be explained in Section \ref{SecConjs}.

The proof of Theorem \ref{ThmMain} uses algebraic geometry, arithmetic of elliptic curves and elliptic surfaces, some results of additive number theory, and some arguments from logic and computability theory. Thus, for the convenience of the reader, we will provide the necessary background and references on each of these topics.


\section{Schnirelmann's density}

\subsection{Densities} Let $\N=\{0,1,2,...\}$ be the set of non-negative integers. Let $S\subseteq \N$ be a set. The \emph{lower density} of $S$ is
$$
\delta_*(S)=\liminf_{x\to \infty} \frac{1}{x}\cdot \# \{n\in S : n\le x\}
$$
and the \emph{upper density} of $S$ is
$$
\delta^*(S)=\limsup_{x\to \infty} \frac{1}{x}\cdot \# \{ n\in S : n\le x\}.
$$
Both quantities are well-defined and they satisfy $0\le \delta_*(S)\le \delta^*(S)\le 1$. If they are equal, their common value is denoted by $\delta(S)$ and this number is called the \emph{natural density} of $S$ in $\N$.

There is another notion of density that will be useful for us. The \emph{Schnirelmann density} of $S$ is
$$
\sigma(S)=\inf_{k\ge 1}\frac{1}{k}\cdot \#\{n\in S : 1\le n\le k\}
$$
where $k$ varies over the positive integers. For instance, if $S=\{2n : n\in \N\}$ is the set of even non-negative integers we have $\sigma(S)=0$, while if $S$ is the set of odd positive integers we have $\sigma(S)=1/2$. In particular, the Schnirelman density behaves quite differently from the lower, upper, and natural densities. Nevertheless we have the following simple observation:

\begin{lemma}\label{LemmaDensities}
Let $S\subseteq \N$. If $\delta_*(S)>0$ and $1\in S$ then $\sigma(S)>0$.
\end{lemma}


\subsection{Schnirelmann's theorem}

Given $S_1,S_2\subseteq \N$ we write $S_1+S_2=\{a+b : a\in S_1\mbox{ and }b\in S_2\}$. For a set $S\subseteq \N$ and a positive integer $h\ge 1$ we write $hS=S+S+...+S$ with $S$ repeated $h$ times. We say that $S\subseteq \N$ is \emph{an additive basis of finite order} if there is $h\ge 1$ with $hS=\N$; note that in this case necessarily $0\in S$. One of the main motivations to study the Schnirelmann density is the following result.

\begin{theorem}[Schnirelmann]\label{ThmSch} Let $S\subseteq \N$ with $0\in S$ and $\sigma(S)>0$. Then $S$ is an additive basis of finite order.
\end{theorem}

We refer the reader to Section 11.3 in \cite{Nathanson} for a proof of this theorem.


\subsection{A criterion for $\Q$ being Diophantine in a $\Q$-algebra}

\begin{lemma}\label{LemmaCriterion} Let $A$ be a (commutative, unitary)  $\Q$-algebra. Suppose that there is a set $T\subseteq \Q$ with the following properties
\begin{itemize}
\item[(i)] $T$ is Diophantine over $A$, and
\item[(ii)] The set $S=T\cap \N$ satisfies $\delta_*(S)>0$.
\end{itemize}
Then $\Q$ is Diophantine in $A$.
\end{lemma}
\begin{proof} Let $T'=\{0,1\}\cup T$ and $S'=T'\cap \N=\{0,1\}\cup S$. By Lemma \ref{LemmaDensities} we have $\sigma(S')>0$, and since  $0\in S'$ Theorem \ref{ThmSch} gives that $S'$ is an additive basis of finite order. Thus, from $T'$ we can construct a subset $T^*\subseteq \Q$ which is Diophantine over $A$ and contains $\N$. The result follows by taking negatives and fractions.  
\end{proof}


\section{Preliminaries on elliptic surfaces}


\subsection{Elliptic surfaces}\label{SecEll}

We refer the reader to \cite{SchuettShioda} for a general introduction and reference on the topic of elliptic surfaces.

Let $k$ be a field of characteristic different from $2$ and $3$. An \emph{elliptic surface} over $\Pro^1$ defined over $k$ is a morphism $\pi\colon X\to \Pro^1$ along with a distinguished section $\sigma_0\colon\Pro^1\to X$, both defined over $k$, satisfying that $X$ is a smooth projective surface and the generic fibre of $\pi$, denoted by $X_\eta$, is an elliptic curve over $k(\Pro^1)$ with neutral element induced by $\sigma_0$. Furthermore, we require that $\pi\colon X\to \Pro^1$ be relatively minimal: There are no $(-1)$-curves  contained in the fibres of $\pi$.

Sections of $\pi$ form an abelian group whose neutral element is $\sigma_0$. This structure is compatible with the group structure of $X_\eta$.

Let $U_\pi\subseteq \Pro^1$ be the locus over which $\pi$ is smooth. Then $U_\pi$ is a non-empty open set of $\Pro^1$ and for all $b\in U_\pi$ the fibre $X_b$ is an elliptic curve defined over the residue field at $b$, with neutral element $\sigma_0(b)$. 

For each of the finitely many points $b\in\Pro^1(k^{\alg})$ not in $U_\pi$, the fibre $X_\pi$ is singular. Kodaira gave a complete classification of the possible singular fibres of an elliptic surface. These come in two types: multiplicative or additive. We will be mostly concerned with the multiplicative case, which corresponds to fibres of type $I_n$ ($n\ge 1$) in Kodaira's classification, and in the additive case we will only need to consider fibres of type $I_n^*$. Fibres of type $I_n$ have $n$ irreducible components while fibres of type $I_n^*$ have $n+5$ irreducible components. 

Finally, we recall that the underlying surface $X$ can have Kodaira dimension $\kappa(X)$ equal to $-\infty$, $0$, or $1$ (general type $\kappa(X)=2$ does not occur for elliptic surfaces). We will be mostly interested in the case $\kappa(X)=1$.


\subsection{Topological Euler characteristic} Let $\pi\colon X\to \Pro^1$ be an elliptic surface defined over $\C$. Let $e(X)$ be the topological Euler characteristic of $X$. For each $b\in \Pro^1$ let $m_b$ be the number of irreducible components of $X_b$ and define
$$e_b(X) = 
\begin{cases}
0 &\mbox{ if $X_b$ is smooth}\\
m_b &\mbox{ if $X_b$ is multiplicative}\\
m_b+1 &\mbox{ if $X_b$ is additive}.
\end{cases}
$$
We have (cf.\ Theorem 6.10 \cite{SchuettShioda}):

\begin{lemma}\label{LemmaTopE} With the previous notation, $e(X)=\sum_{b\in \Pro^1} e_b(X)$.
\end{lemma}

By Noether's formula one sees that  $e(X)$ is a multiple of $12$ (cf.\ Section 6.7 in \cite{SchuettShioda}). This integer determines the Kodaira dimension of $X$ (cf.\ Section 4.10 of \cite{SchuettShioda}).

\begin{lemma}\label{LemmaKdim} With the previous notation, we have
$$
\kappa(X)=\begin{cases}
-\infty & \mbox{ if }e(X)/12 =1\\
0 & \mbox{ if }e(X)/12 =2\\
1 & \mbox{ if }e(X)/12 \ge 3.
\end{cases}
$$
\end{lemma}


\subsection{Ulmer's conjecture} \label{SecUlmer}

In \cite{Ulmer} Ulmer proposed the following conjecture:

\begin{conjecture}[Ulmer]\label{ConjGeneralUlmer} Let $X$ be a smooth projective surface over $\C$ and let $K_X$ be a canonical divisor on $X$. There is a constant $M$ depending only on $X$ such that for every (possibly singular) rational curve $C\subseteq X$ we have $(C.K_X)\le M$.
\end{conjecture}

Evidence for this conjecture is discussed in \cite{Ulmer}. In particular, when $-K_X$ is big or $K_X=0$ the conjecture is trivial, and when $K_X$ is big (i.e.\ $X$ is of general type) the conjecture follows from a conjecture of Lang asserting that $X$ has at most finitely many rational curves.

Let $\pi\colon X\to \Pro^1$ be an elliptic surface with Kodaira dimension $\kappa(X)=1$. Then there is a positive integer $n\ge 1$ such that $K_X$ is linearly equivalent to $nF$ with $F$ a fibre of $\pi$. In this way we see that Conjecture \ref{ConjGeneralUlmer} implies Conjecture \ref{Conj2}.

It is worth pointing out that Conjecture \ref{Conj2} holds in a strong form for a very general elliptic surface $\pi\colon X\to \Pro^1$ over $\C$ with $\kappa(X)=1$. See \cite{Ulmer}.

For later reference, let us say that an elliptic surface $\pi\colon X\to \Pro^1$ defined over $\C$ has the \emph{Ulmer property with bound $M$} if for every rational curve $C\subseteq X$ not contained in a fibre of $\pi$ we have that $(C.F)<M$, where $F$ is a fibre of $\pi$. Then Conjecture \ref{Conj2} can be formulated by saying that $\pi\colon X\to \Pro^1$ has the Ulmer property for some bound $M$ whenever $\kappa(X)=1$.


\section{Elliptic curves}


\subsection{Root numbers and ranks} Let $E$ be an elliptic curve over $\Q$. For each place $v$ of $\Q$ one defines a \emph{local root number} $w_v(E)\in\{-1,1\}$ according to local data of $E$, and the (global) \emph{root number} is $w(E)=\prod_v w_v(E)$. The product is well-defined because $w_p(E)=1$ when $p$ is a prime of good reduction for $E$.

Let $L(E,s)$ be the $L$-function of $E$. By the modularity theorem \cite{Wiles, TaylorWiles, BCDT} this $L$-function is modular of weight $2$, hence it is entire and it satisfies a functional equation relating $s$ to $2-s$. More precisely, let $N$ be the conductor of $E$. Then the completed $L$-function
$$
\Lambda(E,s)=N^{s/2}(2\pi)^{-s}\Gamma(s)L(E,s)
$$
satisfies the functional equation
$$
\Lambda(E,2-s)=w(E)\cdot \Lambda(E,s)
$$
which gives an alternative characterization of the root number. In particular, the previous functional equation reveals that the integer
$$
\rho(E)=\ord_{s=1}L(E,s)
$$
is even if $w(E)=1$ and odd if $w(E)=-1$.

The Birch and Swinnerton-Dyer conjecture predicts that $\rho(E)=\rk E(\Q)$. In particular, one expects the following consequence:

\begin{conjecture}[Parity conjecture] We have that $\rk E(\Q)$ is even if $w(E)=1$ and it is odd if $w(E)=-1$.
\end{conjecture}

What we need is an even weaker consequence:

\begin{conjecture}\label{ConjWeakRoot} If $w(E)=-1$ then $\rk E(\Q)>0$.
\end{conjecture}


\subsection{Equidistribution of root numbers}

In \cite{Helfgott}, Helfgott gives strong support for the following conjecture.

\begin{conjecture}[Averages of root numbers]\label{ConjAvg1} Let $\pi\colon X\to \Pro^1$ be an elliptic surface defined over $\Q$. Fix a choice of affine chart $\A^1\subseteq \Pro^1$ with affine coordinate $t$. Suppose that $\pi$ has a fibre of multiplicative reduction on $\A^1$. Then
$$
\lim_{X\to \infty}\frac{1}{X}{\sum}'_{1\le n\le X} w( X_n)=0
$$
where ${\sum} ' $ indicates that we only consider integers $n$ such that $\pi$ has a smooth fibre at $t=n$.
\end{conjecture}

Indeed, Helfgott \cite{Helfgott} shows that if certain conjectures in analytic number theory hold, then a strong version of Conjecture \ref{ConjAvg1} holds. The aforementioned conjectures in analytic number theory concern the multiplicative structure of polynomial values and are known in some cases, which allowed Helfgott to prove Conjecture \ref{ConjAvg1} for several elliptic surfaces.

Let us reformulate Conjecture \ref{ConjAvg1} in an equivalent form which will be more convenient for us.

\begin{conjecture}\label{ConjAvg2} Let $\pi\colon X\to \Pro^1$ be an elliptic surface defined over $\Q$. Fix a choice of affine chart $\A^1\subseteq \Pro^1$ with affine coordinate $t$. Suppose that $\pi$ has a fibre of multiplicative reduction on $\A^1$. For $\epsilon\in \{-1,1\}$ define
$$
S_\epsilon =\{n\in \N : \pi \mbox{ has a smooth fibre at }t=n\mbox{ and }w(X_n)=\epsilon\}\subseteq \N.
$$
Then both $S_{-1}$ and $S_1$ have natural density $1/2$ in $\N$.
\end{conjecture}


\subsection{Positive rank fibres} \label{SecPositive}

Let us record here the following observation.

\begin{lemma}\label{LemmaImply} Conjectures \ref{ConjWeakRoot} and \ref{ConjAvg2} imply Conjecture \ref{Conj1}.
\end{lemma}

In view of the existing support for Conjecture \ref{ConjWeakRoot} coming from analytic number theory, and for Conjecture \ref{ConjAvg2} coming from the literature around the parity conjecture (and, more generally, the Birch and Swinnerton-Dyer conjecture), Lemma \ref{LemmaImply} gives strong evidence towards Conjecture \ref{Conj1}.


\section{Rational functions}


\subsection{Ramification of rational functions} For a field $k$ and a rational function $f\in k(t)$ we let $\deg(f)$ be the degree of the morphism $f\colon \Pro^1\to \Pro^1$.

Let $d\ge 1$. By writing a rational function as the quotient of two polynomials and considering the coefficients of these polynomials, we see that there is an (irreducible) algebraic variety $\Rat_d$ defined over $\Q$ of dimension $2d+1$ such that for every field $k$ of characteristic $0$ we have that the $k$-rational points of $\Rat_d$ naturally correspond to rational functions $f\in k(t)$ with $\deg(f)=d$. A first observation is:
\begin{lemma}\label{LemmaRat1} We have that $\Rat_d(\Q)$ is Zariski dense in $\Rat_d$. 
\end{lemma}
Let us assume that $k$ has characteristic $0$ and let $f\in k(t)$ be a non-constant rational function of degree $d$.   The rational function $f$ will be called \emph{mildly ramified} if it has $2d-2$ different branch points in $\Pro^1(k^{\alg})$. 
\begin{lemma}\label{LemmaRat2} The locus of all mildly ramified rational functions of degree $d$ in $\Rat_d$ is a non-empty Zariski open set.
\end{lemma}
\begin{proof} Using resultants one sees that the locus of rational functions that are not mildly ramified is Zariski closed in $\Rat_d$. It is properly contained in $\Rat_d$ because there are plenty of mildly ramified maps over $\C$, see \cite{Goldberg}.
\end{proof}
The consequence of Lemmas \ref{LemmaRat1} and \ref{LemmaRat2} that will be relevant for us is the following.

\begin{corollary}\label{CoroExistencef} For each $d\ge 1$ there is a mildly ramified $f\in \Q(t)$ of degree $d$.
\end{corollary}


\subsection{Constructing curves of positive genus} 

\begin{lemma}\label{LemmaMildly} Let $k$ be a field of characteristic $0$ and let $r\ge 1$. Let $f,g\in k(t)$ having degrees $p=\deg(f)$ and $q=\deg(g)\ge 2$ with $p$ prime and $p\ge q+r+1$. Suppose that $f$ is mildly ramified. Then the equation $f(x)=g(y)$ defines a reduced and geometrically irreducible curve in $\A^2$ of geometric genus $\gfrak\ge r$.
\end{lemma}
\begin{proof} We may assume that $k$ is algebraically closed. Let $C$ be the fiber product of the morphisms $f,g\colon\Pro^1\to \Pro^1$, then the equation $f(x)=g(y)$ defines an affine birational model for $C$. We let $\pi_f,\pi_g\colon C\to \Pro^1$ be the corresponding projections of this fiber product satisfying $f\circ \pi_f=g\circ \pi_g$.

Since $p$ is prime the map $f$ is indecomposable, and since $p>q$ we see that the function field extensions of $k(t)=k(\Pro^1)$ induced by $f$ and $g$ are linearly disjoint. It follows that $C$ is reduced and geometrically irreducible, and that $\deg(\pi_f)=q$ and $\deg(\pi_g)=p$.

Let $\nu\colon\widetilde{C}\to C$ be the normalization, let $h=\pi_g\circ \nu\colon\widetilde{C}\to \Pro^1$ and note that $\deg(h)=\deg(\pi_g)=p$. The maps that we have constructed so far are summarized in the following commutative diagram:

$$
\begin{tikzcd}
\widetilde{C} \arrow{rd}{h} \arrow{r}{\nu} & C \arrow{d}{\pi_g} \arrow{r}{\pi_f}& \Pro^1 \arrow{d}{f} \\
  & \Pro^1 \arrow{r}{g} & \Pro^1  .
\end{tikzcd}
$$

Since $f$ has $2p-2$ branch points and $g$ has at most $2q-2$ branch points, we have that $g$ is unramified over at least $2(p-q)$ branch points of $f$. Hence, $h$ has at least $2q(p-q)$ branch points. 

The genus of $\widetilde{C}$ is $\gfrak$. Applying the Riemann-Hurwitz formula to $h$ we deduce 
$$
2\gfrak-2 \ge  -2p + 2q(p-q).
$$ 
Hence, 
$$
\begin{aligned}
\gfrak&\ge q(p-q)-p+1 \\
&= p(q-1)-q^2+1 \\
&\ge (q+r+1)(q-1)-q^2+1 \\
&= rq-r\ge r.
\end{aligned}
$$
\end{proof}


\section{Convenient elliptic surfaces} \label{SecConvenientSurf}


\subsection{Preserving rank $0$}

\begin{lemma}\label{LemmaErk0} Let $\pi\colon X\to \Pro^1$ be an elliptic surface over $\C$ such that $X$ contains only finitely many rational curves. Choose an affine chart $\A^1\subseteq \Pro^1$ and let $t$ be an affine coordinate on $\A^1$. Choose a Weierstrass equation
$$
y^2=x^3+A(t)x+B(t)
$$
for $\pi\colon X\to \Pro^1$ with $A,B\in \C[t]$. Let $f(z)\in \C(z)$ be a non-constant rational function and consider the Weierstrass equation over $\C(z)$
$$
y^2=x^3+A(f(z))x+B(f(z)).
$$
Then this Weierstrass equation defines an elliptic curve $E_f$ over $\C(z)$ with Mordell-Weil rank $0$.
\end{lemma}
\begin{proof} Let $\pi'\colon X'\to \Pro^1$ be the elliptic surface associated to $E_f$, then $\pi'\colon X'\to \Pro^1$ is obtained from $\pi\colon X\to \Pro^1$ by base change by $f\colon \Pro^1\to \Pro^1$. Note that $\C(z)$-rational points of $E_f$ yield sections of $\pi'$ whose image in $X$ via the map $X'\to X$ are rational curves. Each rational curve on $X$ obtained by this procedure can come from at most finitely many points of $E_f(\C(z))$ (in fact, at most $\deg(f)$ of them), hence the result. 
\end{proof}


\subsection{An auxiliary elliptic surface}

From Lemma \ref{LemmaErk0} we deduce:
\begin{lemma} Suppose that Conjecture \ref{ConjAdHoc} holds. There is an elliptic surface $\pi\colon X\to \Pro^1$ defined over $\Q$ such that if we choose an affine chart $\A^1\subseteq \Pro^1$ with affine coordinate $t$, and if we choose a Weierstrass equation
$$
y^2=x^3+A(t)x+B(t)
$$
with $A,B\in \Q[t]$, then the following holds:
\begin{itemize}
\item[(i)] $\pi$ has a multiplicative fibre over $\A^1$;
\item[(ii)] For every non-constant $f\in \Q(z)$, the Weierstrass equation 
$$
y^2=x^3+A(f(z))x+B(f(z))
$$
defines an elliptic curve $E_f$ over $\Q(z)$ with Mordell-Weil rank $0$.
\end{itemize}
\end{lemma}


\section{The Legendre family} \label{SecLegendre}


\subsection{$n$-th power base change} The \emph{Legendre family} is the elliptic surface determined over $\C$ by the  Weierstrass equation
$$
y^2=x(x+1)(x+t)
$$
where $t$ is an affine coordinate on $\A^1\subseteq \Pro^1$. We denote it by $\pi_1\colon X^{(1)}\to \Pro^1$. It is non-isotrivial, i.e.\ the $j$-invariant is not constant. 

 Let $d\ge 1$. Consider the $d$-th power map $\tau_d\colon\Pro^1\to \Pro^1$ given by $\tau_d(t)=t^d$. Let $\pi_d\colon X^{(d)}\to \Pro^1$ be the elliptic surface obtained from $\pi_1\colon X^{(1)}\to \Pro^1$ by base change by $\tau_d$. This elliptic surface admits the Weierstrass equation
$$
y^2=x(x+1)(x+t^d).
$$
The elliptic surface $\pi_d\colon X^{(d)}\to \Pro^1$ is studied in detail in \cite{UlmerLegendre}; let us recall two basic facts that will be useful for us.

\begin{lemma}[cf.\ Section 7 in \cite{UlmerLegendre}]\label{LemmaBadFibres} The elliptic surface $\pi_d\colon X^{(d)}\to \Pro^1$ has a multiplicative fibre of type $I_{2d}$ at $t=0$ and a multiplicative fibre of type $I_2$ at each root of $t^d-1$. At $t=\infty$ the fibre is additive of type $I_{2d}^*$ if $d$ is odd, and multiplicative of type $I_{2d}$ if $d$ is even. There is no other singular fibre.
\end{lemma}

\begin{lemma}[cf.\ Corollary 11.3 in \cite{UlmerLegendre}]\label{LemmaRk0} For each $d\ge 1$, the group of sections of $\pi_d\colon X^{(d)}\to \Pro^1$ is finite.
\end{lemma}

\subsection{Kodaira dimension}

Recall (cf.\ Section \ref{SecEll}) that for an elliptic surface the multiplicative fibres of type $I_n$ have $n$ irreducible components, while the additive fibres of type $I_n^*$ have $n+5$ irreducible components. From this and Lemmas \ref{LemmaTopE} and \ref{LemmaBadFibres} we obtain a formula for the topological Euler characteristic of $X^{(d)}$.

\begin{lemma} The topological Euler characteristic of $X^{(d)}$ is
$$
e(X^{(d)})=\begin{cases}
6d & \mbox{ if $d$ is even};\\
6(d+1) & \mbox{ if $d$ is odd}.
\end{cases}
$$
\end{lemma}

In particular, from Lemma \ref{LemmaKdim} we obtain

\begin{corollary}\label{Corok1} The Kodaira dimension of $X^{(d)}$ is
$$
\kappa(X^{(d)})=\begin{cases}
-\infty & \mbox{ if }d=1,2;\\
0 & \mbox{ if }d=3,4;\\
1 & \mbox{ if }d\ge 5.
\end{cases}
$$
\end{corollary}

\subsection{Mildly ramified base change} Let $f\in \C(t)$ be non-constant and consider the elliptic surface $\pi_{d,f}\colon X^{(d)}_f\to \Pro^1$ obtained from $\pi_{d}\colon X^{(d)}\to \Pro^1$ by base change by $f$. It admits the Weierstrass equation
$$
y^2=x(x+1)(x+f(t)^d).
$$
Recall that for $d\ge 5$ we have $\kappa(X^{(d)})=1$ (cf.\ Corollary \ref{Corok1}).

\begin{lemma}\label{LemmaFiniteness} Let $d\ge 5$ and suppose that $\pi_{d}\colon X^{(d)}\to \Pro^1$ has the Ulmer property for some integer bound $M\ge 2$ (cf.\ Section \ref{SecUlmer}). Let $f(t)\in \C(t)$ be mildly ramified of prime degree $p\ge M+2$. Then $X^{(d)}_f$ only contains finitely many rational curves. 
\end{lemma}
\begin{proof} Let $C\subseteq X^{(d)}_f$ be a rational curve which is not contained in a fibre of $\pi_{d,f}$. Let $u\colon X^{(d)}_f\to X^{(d)}$ be the morphism induced by the base change, so that $\pi_d\circ u= f\circ \pi_{d,f}$. We claim that $u(C)$ is a section of $\pi_d$. 

Suppose, for the sake of contradiction, that $u(C)$ is not a section of $\pi_d$. Note that $u(C)$ is a rational curve; let $\nu\colon\Pro^1\to u(C)$ be its desingularization and let $g=\pi_d\circ \nu\in \C(s)$ where $s$ is an affine coordinate in $\Pro^1$. The maps that we have defined are summarized in the following commutative diagram:

$$
\begin{tikzcd}
C \arrow[hookrightarrow]{r}  & X^{(d)}_f \arrow{r}{u} \arrow{d}{\pi_{d,f}}& X^{(d)} \arrow{d}{\pi_d} & \Pro^1 \arrow{l}{\nu} \arrow{dl}{g} \\
  & \Pro^1 \arrow{r}{f} & \Pro^1.  &
\end{tikzcd}
$$

Since $u(C)$ is not a section of $\pi_d$ and $\nu$ is birational onto $u(C)$, the integer $q=\deg(g)$ is at least $2$ and by our assumption on the Ulmer property we have $q\le M$. By Lemma \ref{LemmaMildly} and the assumption $p\ge M+2$ the equation $f(t)=g(s)$ defines a reduced irreducible curve $V$ of geometric genus $\gfrak\ge 1$. However, since $\pi_d\circ u= f\circ \pi_{d,f}$ and $\nu$ is birational, there is a rational map $C\dasharrow V$ and $C$ has geometric genus $0$; a contradiction.

Finally, by Lemma \ref{LemmaRk0} we see that necessarily $u(C)$ is a torsion section of $\pi_d\colon X^{(d)}\to \Pro^1$ and we obtain finiteness.
\end{proof}


\subsection{Constructing an elliptic surface under Ulmer's conjecture}\label{SecConjs}

\begin{lemma}\label{LemmaConjs} Suppose that Conjecture \ref{Conj2} holds. Then Conjecture \ref{ConjAdHoc} holds.
\end{lemma}
\begin{proof} Fix $d\ge 5$. Assuming Conjecture \ref{Conj2} we have that $\pi_d\colon X^{(d)}\to \Pro^1$ has the Ulmer property for some integer bound $M\ge 2$. By Corollary \ref{CoroExistencef} there is a mildly ramified $f\in \Q(t)$ of prime degree $p\ge M+2$. By Lemma \ref{LemmaFiniteness} the elliptic surface $\pi_{d,f}\colon X^{(d)}_f\to \Pro^1$ satisfies that $X^{(d)}_f$ contains only finitely many rational curves.

Lemma \ref{LemmaBadFibres} shows that $\pi_d$ has fibres of multiplicative reduction, and since $\pi_{d,f}\colon X^{(d)}_f\to \Pro^1$ is a base change of $\pi_{d}\colon X^{(d)}\to \Pro^1$, we conclude that $\pi_{d,f}$ also has fibres of multiplicative reduction. 

Finally, since $\pi_{d,f}\colon X^{(d)}_f\to \Pro^1$ admits the Weierstrass equation
$$
y^2=x(x+1)(x+f(t)^d)
$$
we see that $\pi_{d,f}\colon X^{(d)}_f\to \Pro^1$ can be defined over $\Q$.
\end{proof}

We remark that for Conjecture \ref{ConjAdHoc} to hold it is not necessary to have Conjecture \ref{Conj2} in full generality; it suffices that for some $d\ge 5$ the elliptic surface $\pi_{d}\colon X^{(d)}\to \Pro^1$ has the Ulmer property for some bound $M$.


\section{The main result}


\subsection{Rational torsion} 

\begin{lemma}\label{LemmaTorsion}
Let $E$ be an elliptic curve over $\Q(z)$ and let $P\in E(\Q(z))$ be a torsion point. Then the order of $P$ is at most $12$.
\end{lemma}
\begin{proof} For elliptic curves over $\Q$ the analogous statement is a theorem of Mazur \cite{Mazur}; this proves the result if the elliptic curve $E$ is actually defined over $\Q$. For non-constant elliptic curves over $\Q(z)$ this follows from injectivity of specialization of elliptic surfaces, see \cite{Silverman}. 
\end{proof}


\subsection{Defining $\Q$ in $\Q(z)$} 

\begin{theorem}\label{ThmV1} Assume Conjectures \ref{Conj1} and \ref{ConjAdHoc}. Then $\Q$ is Diophantine in $\Q(z)$.
\end{theorem}
\begin{proof} Assuming Conjecture \ref{ConjAdHoc}, we can take an elliptic surface $\pi\colon X\to \Pro^1$ with Weierstrass equation
$$
y^2=x^3+A(t)x+B(t)
$$
with $A,B\in \Q[t]$, as the one afforded by Lemma \ref{LemmaErk0}. Let $\Sigma\subseteq \Q$ be the finite set of rational numbers $c$ such that
$$
y^2=x^3+A(c)x+B(c)
$$
is singular. For $f\in \Q(z)-\Sigma$ (possibly constant) let $E_f$ be the elliptic curve over $\Q(z)$ defined by
$$
y^2=x^3+A(f)x+B(f).
$$
Define the set
$$
\Rcal =\{f\in \Q(z)-\Sigma : \rk E_f(\Q(z))>0  \}\subseteq \Q(z).
$$
We claim that $\Rcal$ is Diophantine. Indeed,  $\rk E_f(\Q(z))>0$ if and only if there are $u,v\in \Q(z)$ such that $v^2=u^3+A(f)u+B(f)$ and $(u,v)$ is not a torsion point of $E_f$ of order at most $12$ (cf.\ Lemma \ref{LemmaTorsion}). The latter is in fact a list of $12$ Diophantine conditions, which can be seen by using division polynomials of $E_f$ and the fact that $\ne$ is Diophantine over any field. This proves that $\Rcal$ is Diophantine in $\Q(z)$.

Since $\rk E_f(\Q(z))=0$ whenever $f$ is non-constant (recall that $\pi\colon X\to \Pro^1$ is produced by Lemma \ref{LemmaErk0}) we conclude that, in fact, $\Rcal\subseteq \Q$.

By the Riemann-Hurwitz formula, we have that for $c\in \Rcal\subseteq \Q$ all the $\Q(z)$-rational points of $E_c$ are actually $\Q$-rational, for otherwise we would get a non-constant map from $\Pro^1$ to $E_c$. Therefore
$$
\Rcal = \{c\in \Q-\Sigma : \rk E_c(\Q)>0\}.
$$
Let $S=\Rcal\cap \N$. Since $\pi\colon X\to \Pro^1$ has a multiplicative fibre over $\A^1$, Conjecture \ref{Conj1} gives that $\delta_*(S)>0$. We conclude by Lemma \ref{LemmaCriterion}.
\end{proof}


\subsection{All listable sets in $\Q$} In view of Lemma \ref{LemmaConjs}, the following result implies Theorem \ref{ThmMain}.

\begin{theorem}\label{ThmV2} Assume Conjectures \ref{Conj1} and \ref{ConjAdHoc}. Then every listable set $T\subseteq \Q$ is Diophantine in $\Q(z)$.
\end{theorem}
\begin{proof}
By Theorem \ref{ThmV1} we have that $\Q$ is Diophantine in $\Q(z)$. Thus, it follows from the results in Section 3 of \cite{Denef} that $\Z$ is Diophantine in $\Q(z)$. From general principles (see for instance Section 4 in \cite{Pasten}) we know that listable subsets of $\Q$ are positive existentially definable over $\Q$ in the language $\{0,1,+,\times, =, \Z\}$. The result follows.
\end{proof}


\section{Acknowledgments}

We thank Thanases Pheidas for encouraging us to work on this problem and we thank Jerson Caro for his feedback on a first version of this manuscript. We are grateful to Arno Fehm for answering several questions.

N.G.-F. was supported by ANID Fondecyt Regular grant 1211004 from Chile.

H.P. was supported by ANID (ex CONICYT) Fondecyt Regular grant 1190442 from Chile.

This material is based upon work supported by the National Science Foundation under Grant No. DMS-1928930 while the authors participated in the program Definability, Decidability, and Computability in Number Theory, part 2,  hosted by the Mathematical Sciences Research Institute in Berkeley, California, during the Summer of 2022.



\begin{thebibliography}{9}         

\bibitem{BCDT} C. Breuil, B. Conrad, F. Diamond, R. Taylor, \emph{On the modularity of elliptic curves over $\mathbb{Q}$: wild $3$-adic exercises}. J. Amer. Math. Soc. 14 (2001), no. 4, 843-939. 

\bibitem{DaansThesis} N. Daans, \emph{Existential first-order definitions and quadratic forms}. Thesis, Universiteit Antwerpen (2022)

\bibitem{Denef} J. Denef, \emph{The Diophantine problem for polynomial rings and fields of rational functions}. Trans. Amer. Math. Soc. 242 (1978), 391-399.

\bibitem{FehmGeyer} A. Fehm, W.-D. Geyer, \emph{A note on defining transcendentals in function fields.} J. Symbolic Logic 74 (2009), no. 4, 1206-1210. 

\bibitem{Goldberg} L. Goldberg, \emph{Catalan numbers and branched coverings by the Riemann sphere}.  Adv. Math. 85 (1991), no. 2, 129-144. 


\bibitem{Helfgott} H. Helfgott, \emph{On the behaviour of root numbers in families of elliptic curves}. Preprint (2009) \url{https://arxiv.org/abs/math/0408141}

\bibitem{KoeLarge} J. Koenigsmann,  \emph{Defining transcendentals in function fields}. J. Symbolic Logic 67 (2002), no. 3, 947-956.

\bibitem{Mazur} B. Mazur, \emph{Modular curves and the Eisenstein ideal}. With an appendix by Mazur and M. Rapoport. Inst. Hautes \'Etudes Sci. Publ. Math. No. 47 (1977), 33-186 (1978).

\bibitem{Nathanson} M. Nathanson, \emph{Elementary methods in number theory}. Graduate Texts in Mathematics, 195. Springer-Verlag, New York, 2000.

\bibitem{Pasten} H. Pasten,  \emph{Notes on the DPRM property for listable structures}. J. Symb. Log. 87 (2022), no. 1, 273-312.

\bibitem{RR} R. Robinson, \emph{The undecidability of pure transcendental extensions of real fields}. Z. Math. Logik Grundlagen Math. 10 (1964), 275-282. 

\bibitem{SchuettShioda} M. Sch\"utt, T. Shioda, \emph{Elliptic surfaces}. Algebraic geometry in East Asia--Seoul 2008, 51-160, Adv. Stud. Pure Math., 60, Math. Soc. Japan, Tokyo, 2010.

\bibitem{Silverman} J. Silverman, \emph{Heights and the specialization map for families of abelian varieties}. J. Reine Angew. Math. 342 (1983), 197-211. 

\bibitem{TaylorWiles} R. Taylor, A. Wiles, \emph{Ring-theoretic properties of certain Hecke algebras}. Ann. of Math. (2) 141 (1995), no. 3, 553-572. 

\bibitem{Ulmer} D. Ulmer, \emph{Rational curves on elliptic surfaces}. J. Algebraic Geom. 26 (2017), no. 2, 357-377.

\bibitem{UlmerLegendre} D. Ulmer, \emph{Explicit points on the Legendre curve}. J. Number Theory 136 (2014), 165-194.

\bibitem{Wiles} A. Wiles, \emph{Modular elliptic curves and Fermat's last theorem}. Ann. of Math. (2) 141 (1995), no. 3, 443-551.

\end{thebibliography}
\end{document}